\documentclass[12pt,a4paper]{amsart}
\usepackage[margin=.8in]{geometry}
\usepackage{graphicx} 
\usepackage{tikz}
\usetikzlibrary{shapes}
\usetikzlibrary{intersections}
\usetikzlibrary{svg.path}
\usetikzlibrary{decorations.markings}
\usetikzlibrary{hobby}
\usepackage{amsmath,mathrsfs}
\usepackage{physics}
\usepackage{amssymb,amsrefs}
\usepackage{bm}
\usepackage{dsfont}
\usepackage{hyperref}

\theoremstyle{plain} 
\newtheorem{theorem}{Theorem}
\newtheorem*{theorem*}{Theorem}
\newtheorem{lemma}[theorem]{Lemma}
\newtheorem*{lemma*}{Lemma}

\newtheorem*{corollary*}{Corollary}
\newtheorem{proposition}[theorem]{Proposition}
\newtheorem*{proposition*}{Proposition}

\newtheorem*{definition*}{Definition}

\newtheorem*{example*}{Exemple}

\newtheorem*{remark*}{Remark}
\newtheorem*{remarks*}{Remarks}

\newcommand{\E}{\mathbb{E}}
\newcommand{\PP}{\mathbb{P}}

\newcommand{\C}{\mathbb{C}}
\newcommand{\CC}{\mathbb{C}}

\newcommand{\DD}{\mathbb{D}}
\renewcommand{\dd}{\mathrm{d}}
\newcommand{\Ov}{\mathscr{O}}
\newcommand{\la}{\lambda}

\newcommand{\beq}{\begin{equation}}
\newcommand{\eeq}{\end{equation}}

\title{On stability of outliers from the circular law}
\author{Guillaume Dubach}
\author{Aniss Fares}
\author{Lilou Thomas}
\address{Centre de Mathématiques Laurent Schwartz, \'Ecole Polytechnique, Institut Polytechnique de Paris, Palaiseau, France}
\email{guillaume.dubach@polytechnique.edu}
\thanks{The authors gratefully acknowledge support from the Fondation de l'Ecole polytechnique, as well as from the Agence Nationale de la Recherche: ANR-25-CE40-5672 and ANR-25-CE40-1380.}

\date{June 2026}

\begin{document}

\begin{abstract}
    This work investigates the stability of outliers from the circular law, via the convergence of their associated diagonal overlaps between eigenvectors -- also known as the squared eigenvalue condition numbers. We consider and compare two paradigmatic cases, namely:
    \begin{enumerate}
    \item The Complex Ginibre Ensemble conditioned on the existence of an outlier;
    \item The outlier induced by a rank-one Hermitian perturbation of a complex Ginibre matrix.
    \end{enumerate}
    In both cases, we prove almost sure convergence towards a specific constant that only depends on the radius of the outlier and its status -- either conditioned or induced. 
    These results can be generalized to other complex integrable ensembles with the same techniques, and complement our understanding of eigenvalue stability in non-Hermitian ensembles.
\end{abstract}

\maketitle


\begin{figure}[h!]
    \centering
    \includegraphics[width=0.5\linewidth]{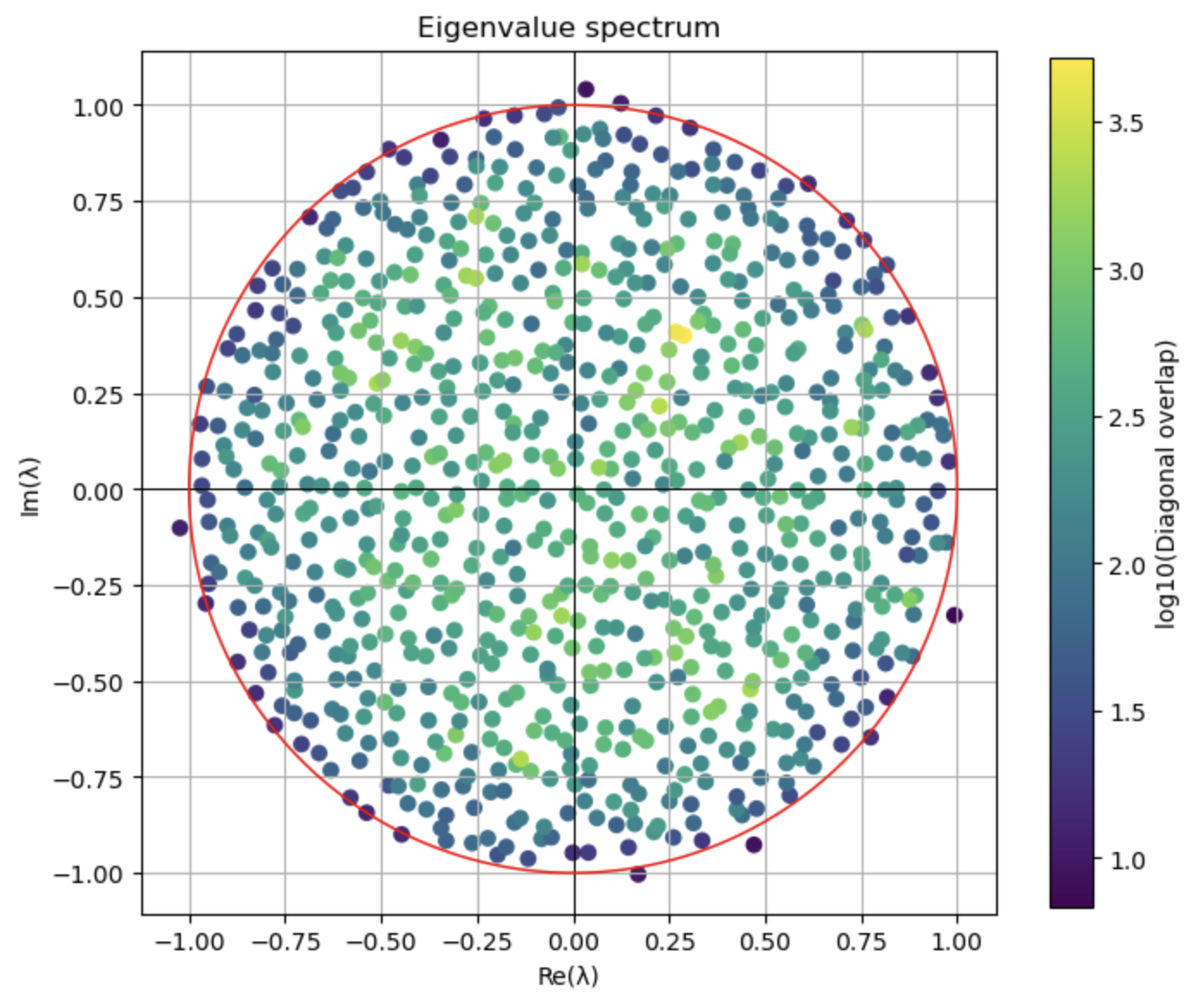}
    \caption{The spectrum of a Ginibre matrix for $N=800$, with the relative size of diagonal overlaps indicated on a color scale. This illustrates a `freezing at the edge' \cite{FreezingEdge} phenomenon, edge eigenvalues being more stable than bulk eigenvalues. The purpose of this work is to prove that outliers are even more stable than edge eigenvalues.}
    \label{fig:spectrum_Ginibre}
\end{figure}


\tableofcontents

\newpage
\section{Introduction}

A key question in random matrix theory is to quantify the stability of an eigenvalue under a perturbation of the matrix; in non-Hermitian ensembles especially, eigenvalues are expected to be less stable -- in various ways -- than in Hermitian settings. Since the seminal works of Chalker and Mehlig \cites{ChalkerMehlig, MehligChalker1, MehligChalker2}, eigenvalue stability in non-normal random matrix ensembles is known to be related to the so-called matrix of \textit{overlaps between eigenvectors}, defined by
\begin{equation}\label{def_overlaps}
    \Ov_{ij} := \langle L_i | L_j \rangle \langle R_j | R_i \rangle,
\end{equation}
where $(L_i)$ and $(R_j)$ are two families of left and right eigenvectors respectively, normalized according to the bi-orthogonality condition $\langle L_i | R_j \rangle = \delta_{ij}$. 
The way in which the matrix of overlaps quantifies eigenvalue stability is for instance made clear by the following fact\footnote{This fact can be proved by a direct computation, which is consistent with other perturbative results appearing in the literature, for classical random perturbations \cites{ChalkerMehlig, MehligChalker1, MehligChalker2} or stochastic matrix evolution \cites{BourgadeDubach,GrelaWarchol}. }:

\begin{proposition}(Gaussian perturbation of eigenvalues \cite{Thomas2026})\label{prop1}
Let $M \in \mathcal{M}_N(\mathbb{C})$ be a matrix with simple spectrum $\{ \la_1, \dots, \la_N \}$ and let $\Ov$ be the matrix of overlaps between eigenvectors of $M$. If $X$ is a matrix with i.i.d. standard complex Gaussian entries, and if $M(t) = M+ t X$, then the eigenvalue trajectories $\{ \la_1(t), \dots, \la_N(t) \}$ are smooth on some neighborhood of $t=0$, and the derivatives at $t=0$ form a Gaussian vector with covariance $\Ov$ : 
    $$ (\lambda_1'(0), ..., \lambda_N'(0)) \overset{d}{=} \mathcal{N}_{\mathbb{C}}(0, \Ov).$$
\end{proposition}

In particular, the stability of the eigenvalue $\la_i$ considered alone is measured by its associated diagonal overlap $\Ov_{ii}$, which is nothing else than the product of the squared $L^2$-norm of its associated left and right eigenvectors:
$$
\Ov_{ii} = \| L_i \|^2 \| R_i \|^2,
$$
an intrinsic quantity that also appears in the literature as the squared \textit{eigenvalue condition number} \cites{Belinschi2017}.
The spectrum of a Ginibre matrix with the relative size of diagonal overlaps is represented on Figure \ref{fig:spectrum_Ginibre}.

This paper investigates the stability -- and thus, the distribution of the diagonal overlap -- of an eigenvalue that lies outside the unit disk: such eigenvalues are usually called outliers. We consider two paradigmatic cases and compare them. The first case is that of the complex Ginibre ensemble conditioned on $\{ \la_1 = z \}$ with $|z|>1$ (i.e. conditioned on a rather untypical event). The second case is that of a rank-one positive-definite perturbation of Ginibre that is taylored to induce an outlier around $z=t$, with $t>1$. \\

\subsection{Organization}
The plan of the paper is as follows:
in section \ref{sec_conditional}, we work with a complex Ginibre matrix, under the condition that $\{\la_1=z\}$ is an outlier. This is the setting of the recent bachelor thesis \cite{Thomas2026}, in which the complex elliptic case is also covered. More precisely, we consider complex Ginibre eigenvalues $\{ \la_1, \dots, \la_N \}$ under the condition that $\{ \la_1 = z \}$, where $z$ is a fixed point in the complex plane with $|z|>1$. We prove the almost sure convergence of the diagonal overlap towards an order one constant:

\begin{theorem}[Conditioned case]\label{thm_conditional}
    Let $G$ be a complex Ginibre matrix and $\Ov$ its matrix of overlaps between eigenvectors. Then, conditionally on $\{\la_1 = z\}$ with $|z|>1$, 
    $$
    \mathscr{O}_{11} \xrightarrow[N \rightarrow \infty]{a.s.} \frac{|z|^2}{|z|^2-1}.
    $$
\end{theorem}
We also verify by a direct computation that the conditional expectation $\E \left( \Ov_{11} | \la_1 = z \right)$ converges to the same limit value, whereas the conditional variance of $\Ov_{11}$ is infinite. Thus, diagonal overlaps of conditioned outliers are random variables of order $1$, but still heavy-tailed at finite $N$. \\

In section \ref{sec_perturbed}, we consider a different setting, namely a rank-one Hermitian perturbation\footnote{This model was originally considered for its applications to neural networks: indeed, it can be used as a toy model of a synaptic matrix \cite{RajanAbbott}.} of the matrix $G$, considering
\begin{equation}\label{eq:rank_one_pert}
G(t) = G + t vv^*, \quad t>1
\end{equation}
with $v$ a unit vector independent on $G$.
It is known \cites{Chafai,Tao2013} that for any $t>1$, with high probability, the spectrum of $G(t)$ has one outlier in a neighborhood of $z = t$; we may assume that this eigenvalue is $\la_1$, so that
$$
\la_1 (t) \simeq t  + O_{\PP}(N^{-\frac12}).
$$
In this context, we obtain another result of almost sure convergence to an order $1$ constant:

\begin{theorem}[Perturbed case]\label{thm_perturbed}
    Let $t>1$ and $\lambda_1(t)$ the outlier eigenvalue of $G(t)$ defined as in eq. \eqref{eq:rank_one_pert}. Then
    $$
    \mathscr{O}_{11} \xrightarrow[N \rightarrow \infty]{a.s.} \left( \frac{t^2}{t^2-1} \right)^2.
    $$
\end{theorem}

It is an interesting fact that the constant obtained in the perturbed case (Theorem \ref{thm_perturbed}) corresponds to the square of the one obtained in the conditional case (Theorem \ref{thm_conditional}). Apart from this slight difference, the phenomenon is the same in both cases, namely: almost sure convergence of the diagonal overlap towards an order one object, as represented on Figure \ref{fig:overlap_distr}.

\begin{figure}[h!]
    \centering
    \includegraphics[width=0.75\linewidth]{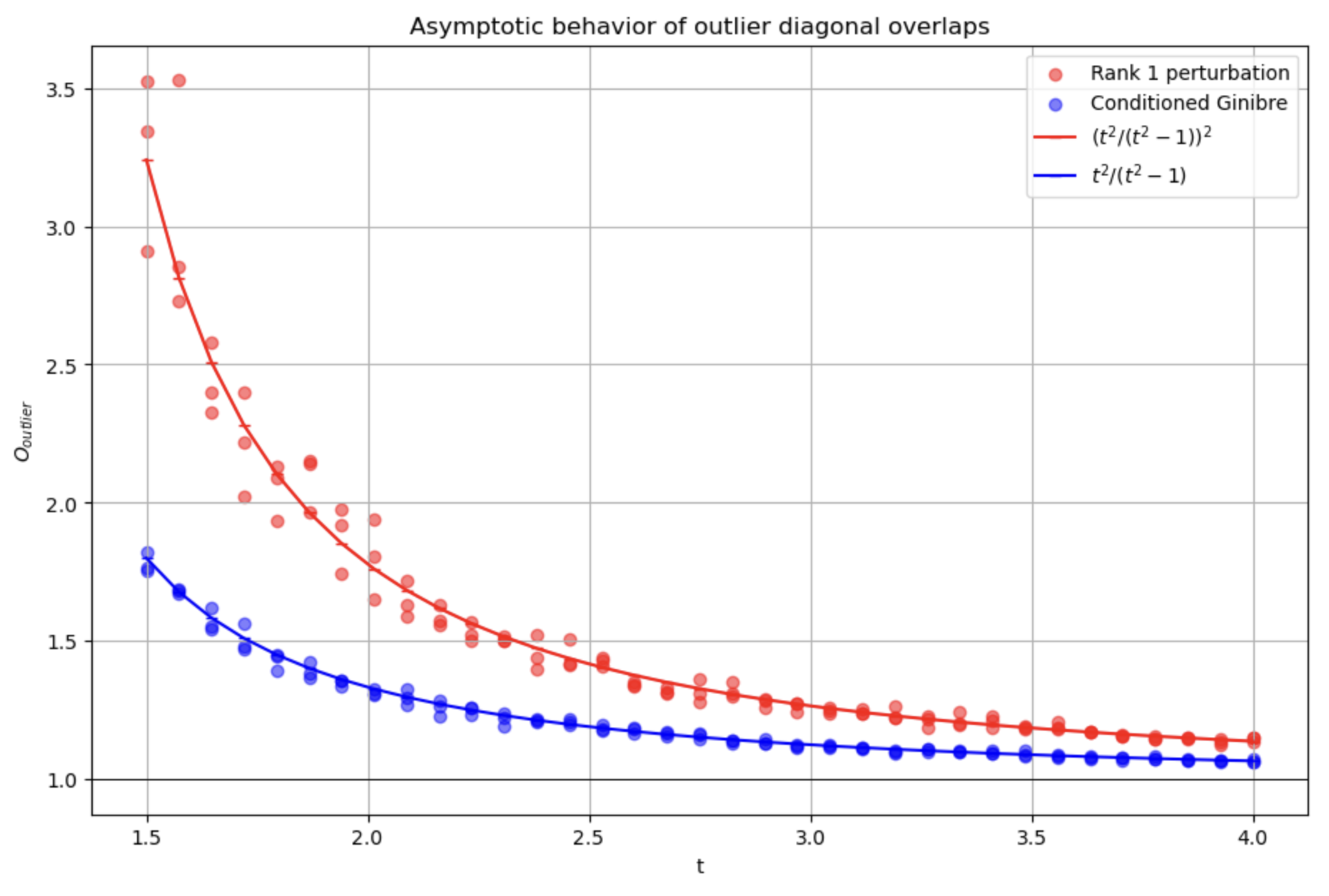}
    \caption{Comparison of the values of diagonal overlaps in the conditioned and perturbed case, for matrices of size $N=400$, over $210$ experiments for values of $t$ ranging from $1.5$ to $4$. The fact that the fluctuations are more important in the perturbed case is likely due to the fluctuations of $\lambda_1$ around its expected value $z=t$.}
    \label{fig:overlap_distr}
\end{figure}

Convergence of diagonal overlaps towards an order $1$ quantity indicates that outlier eigenvalues are much more stable than their bulk (or even edge) counterparts. Indeed, the above results should be contrasted with the convergence in distribution obtained in \cite{BourgadeDubach} when conditioning on $\{ \la_i = z\}$ with $|z| < 1$:
$$
\frac{1}{N} \Ov_{ii} \xrightarrow[N \rightarrow \infty]{d} \frac{1-|z|^2}{\gamma_2}
$$
which means that in the bulk regime, diagonal overlaps are random objects of order $N$ with important fluctuations (the inverse of a Gamma distribution with parameter $2$ is a heavy-tailed distribution, with expectation 1 and infinite variance). At the edge, the expected size of the diagonal overlap is of order $\sqrt{N}$, as shown in \cites{Akemann_determinantal,Thomas2026,Walters, Zhang}. Some heuristic comments on these phenomena and the scope of the available proof techniques are shared in conclusion (Section \ref{sec_conclusion}). 
Finally, the asymptotic expansion of the incomplete gamma function that is needed in section~\ref{sec_conditional} is provided in the appendix.

\subsection{Notations and assumptions.} Throughout the paper, $\lvert \lvert . \lvert \lvert$ denotes either the Euclidean norm if it is applied to a vector or the operator norm induced by the Euclidean norm if it is applied to a matrix, and $\dd^2 z $ stands for integration with respect to the Lebesgue measure on $\CC$. We will denote the partial sums of the exponential by
$$
e^{(m)} (x) = \sum_{k=0}^m \frac{x^k}{k!}.
$$
The joint distribution of rescaled Ginibre eigenvalues is denoted by $\mu_N$, and the measure of the same process conditioned on $z_1 = z$ will be denoted by $\mu_{N-1}^{(z)}$, so that we have
\begin{equation}\label{joint_density_1}
    \dd \mu_N(z_1, \dots, z_N) = \frac{N^{\frac{N(N+1)}{2}}}{ \pi^N \prod_{k=1}^N k!} \prod_{1 \leq i,j \leq N} |z_i - z_j|^2 e^{-\sum_{i=1}^N N |z_i|^2} \dd^2 z_1 \cdots  \dd^2 z_N
\end{equation}
and for the conditioned process:
\begin{equation}\label{joint_density_2}
    \dd \mu_{N-1}^{(z)}(z_2, \dots, z_N) = \frac{1}{\rho_1(z)} \prod_{k=2}^N |z-z_k|^2  e^{- N |z|^2} \prod_{2 \leq i,j \leq N} |\la_i - \la_j|^2 e^{-\sum_{i=2}^N N |z_i|^2} \dd^2 z_1 \cdots  \dd^2 z_N
\end{equation}
where $\rho_1(z)$ is the first correlation function of the complex Ginibre determinantal point process, given by
$$
\rho_1(z) = e^{(N)} (N|z|^2).
$$

In both parts, we will rely on the fact that the outlier is almost surely well separated from the rest of the spectrum. Namely, in both cases, with probability $1$, the rest of the spectrum is in the disk $D(0,1+\epsilon)$ for fixed epsilon, and the outlier eigenvalue is assumed to have radius $|\la_1|> 1+ \epsilon + \delta$, so that
\beq\label{delta_distance}
\forall k, \quad |\la_1 - \la_k| > \delta.
\eeq
For simplicity, we keep this $\delta$ constant in this work; however, most of our results can be improved with minimal effort and proved to hold for smaller outliers, i.e. outliers closer to the unit disk, with $\delta_N \rightarrow 0$ in a reasonable range as $N \rightarrow \infty$.

We also rely on the rest of the spectrum converging weakly to the circular law; in the conditioned case, this is Lemma \ref{lemma:circular_law}, proved in \cite{Thomas2026} from the explicit form of the correlation functions of $\mu_{N-1}^{(z)}$. In the perturbed case, it is a theorem originally proved in \cite{Chafai}.

Some uniform bounds of Section \ref{sec_perturbed} rely on finite nets arguments; a finite $\eta$-net for a given compact $K$ is a finite set $S$ of points such that every $x\in K$ is at distance at most $\eta$ from $S$, i.e.
$$
\forall z \in K \quad \exists x \in S, \quad |x-z| < \eta.
$$

\newpage
\section{Stability of a conditioned outlier}\label{sec_conditional}

In this section, we consider the Complex Ginibre Ensemble \textit{conditioned} on having an outlier -- that is to say that we simply assume that the outlier happens to be there, despite the relative rareness of such an event. Without loss of generality, we will assume that the outlier is the eigenvalue $\la_1$, and study its diagonal overlap $\Ov_{11}$, for which we prove convergence almost surely (Theorem \ref{thm_cond_as_conv}), and in expectation (Proposition \ref{prop_cond_exp_conv}), towards the order one constant $\frac{|z|^2}{|z|^2-1}$. \\

\begin{figure}[h!]
    \centering
    \includegraphics[width=0.9\linewidth]{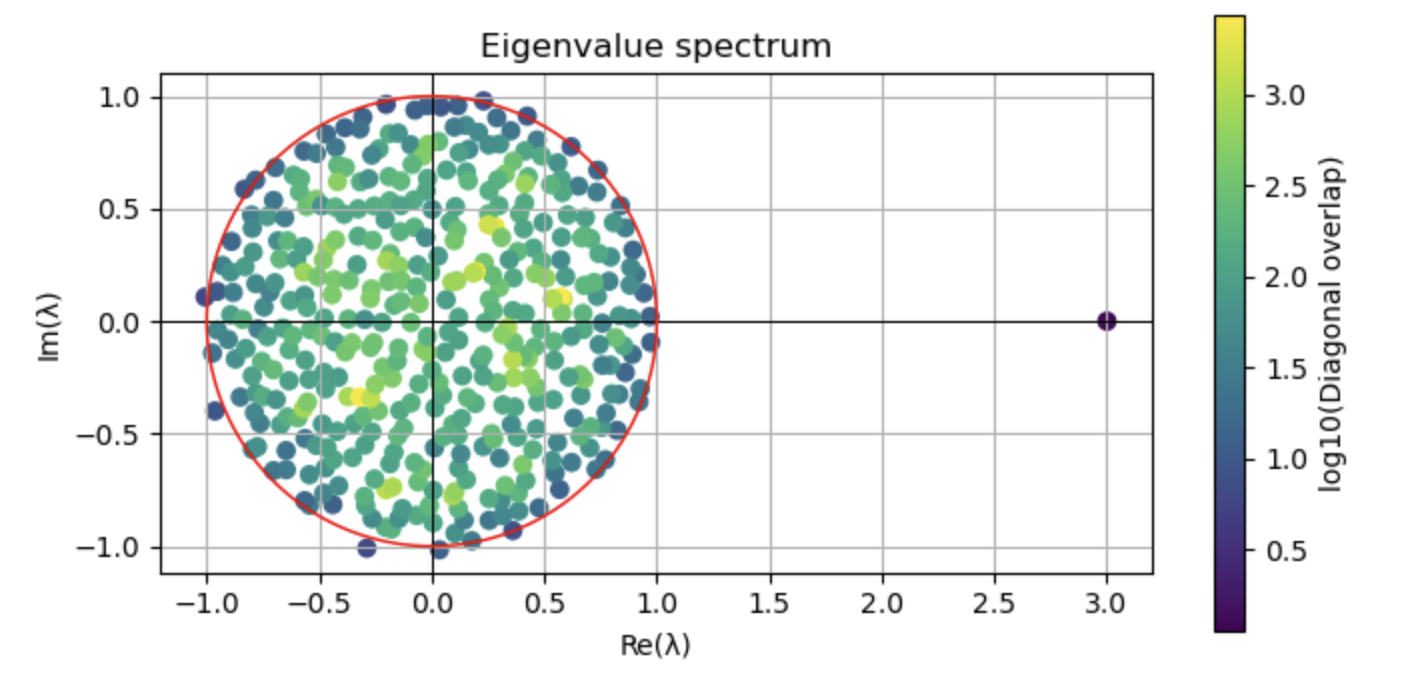}
    \caption{The spectrum of a Ginibre matrix conditioned on $\la_1 = 3$, with the relative size of diagonal overlaps indicated on a color scale.}
    \label{fig:spectrum_cond}
\end{figure}

We will rely on the following lemmata, which can all be proved by classical techniques.

\begin{lemma}[Gaussian tail estimate]\label{lemma:Gaussian_tail}
Let $(X_k^{(N)})_{1 \leq k \leq N}$ be an array of i.i.d. complex standard Gaussian variables. Then, for every $\gamma > 2$, almost surely, there exists an integer $N_0$ such that
$$
\forall N \geq N_0, \ \forall k \in [\![1,N]\!], \ \left| X_k^{(N)} \right|^2 < \gamma \log N.
$$
\end{lemma}

\begin{lemma}[Second largest eigenvalue \cite{Thomas2026}]\label{lemma:spectral_radius}
Let $\{ \la_k^{(N)} \}_{1 \leq k \leq N}$ be an array of random complex numbers such that for every $N$, $\{ \lambda_1^{(N)}, \dots, \lambda_N^{(N)} \}$ is distributed like the spectrum of a complex Ginibre matrix conditioned on $\{ \lambda_1 = z \}$ with fixed $z \in \CC$; then almost surely, for any fixed $\epsilon > 1$ there exists an integer $N_0 >0$ such that
$$
\forall N \geq N_0, \  \forall k \in [\![ 2,N ]\!] \ |\la_k| < 1+\epsilon.
$$
\end{lemma}

\begin{lemma}[Conditional Circular Law \cite{Thomas2026}]\label{lemma:circular_law}
Let $\{ \lambda_1, \dots, \lambda_N \}$ the spectrum of a complex Ginibre matrix conditioned on $\{ \lambda_1 = z \}$ with fixed $z \in \CC$; almost surely, the empirical spectral distribution
$$
\frac{1}{N-1} \sum_{k=2}^N \delta_{\la_k}
$$
converges weakly to the circular law. That is to say that, for any continuous bounded function~$\phi$, we have
$$
\frac{1}{N-1} \sum_{k=2}^N \phi(\lambda_k) 
\xrightarrow[]{N \rightarrow \infty} 
\frac{1}{\pi} \int_{\DD} \phi(z) \dd z.
$$
\end{lemma}

Note that the conditioned Ginibre process with joint density \eqref{joint_density_2} is a determinantal process that resembles the unconditioned Ginibre process, with a slightly modified kernel. The bound on the second largest eigenvalue was readily justified in \cite{Thomas2026} with a variation of Kostlan's argument \cite{Kostlan}, and the expected circular law was obtained through the classical kernel estimates from Ameur, Hedenmalm and Makarov \cite{AHM1}. \\

With the above facts in mind, we can provide a proof of our main theorem.

\begin{theorem}[Almost-sure convergence of diagonal overlaps]\label{thm_cond_as_conv}
Let $G$ be a complex Ginibre matrix conditioned on $\{ \la_1 = z \}$ with $|z| > 1$. Then, the diagonal overlap associated to $\la_1$ converges almost surely to a constant, namely:
$$
\Ov_{11} \xrightarrow[N \rightarrow \infty]{a.s.} \frac{|z|^2}{|z|^2-1}.
$$
\end{theorem}

\begin{proof}
The distribution of the diagonal overlap $\mathscr{O}_{11}$ was investigated in \cite{BourgadeDubach}, and in particular it was found (Theorem 2.2) that
\begin{equation}\label{eq_distr}
\Ov_{11} \stackrel{d}{=} \prod_{k=2}^N \left( 1 + \frac{|X_k|^2}{N |\la_1 - \la_k|^2} \right),
\end{equation}
where $\la_1, \dots, \la_N$ are the eigenvalues of $G$, and $X_1, \dots, X_k$ are i.i.d. standard complex Gaussian variables. The probability measures associated to these families of variables are denoted by $\PP_{\la}$ and $\PP_{X}$ respectively. 
Taking the logarithm of \eqref{eq_distr} yields
\begin{equation}\label{eq_log_distr}
\ln \Ov_{11} \stackrel{d}{=} \sum_{k=2}^N \ln \left( 1 + \frac{|X_k|^2}{N |\la_1 - \la_k|^2} \right).
\end{equation}
We then use the following bound on the logarithm:
$$
\forall c > \frac12, \quad 
\exists \eta >0 , \quad 
\forall x \in [- \eta, \eta] \quad 
\left| \ln \left( 1 + x \right) - x \right| < c |x|^2.
$$
For our current purpose, any choice of $c$ and $\eta$ will do; for instance one may take $c = 1$ and $\eta = \frac12$. Indeed, thanks to lemmata \ref{lemma:Gaussian_tail} and \ref{lemma:spectral_radius}, we know that there are constants $\gamma > 2$ and $\delta >0$ such that with probability $1$, for $N$ large enough,
$$
\forall k = 1, \dots, N \qquad
|X_k|^2 < \gamma \log N \
\quad \text{and} \quad
\quad |\la_1 - \la_k| > \delta
$$
in particular, with probability $1$, for large enough $N$, the term $\frac{|X_k|^2}{N |\la_1 - \la_k|^2}$ in \eqref{eq_log_distr} is smaller than any prescribed $\eta$, so that we can write:
$$
\left| \ln \Ov_{11} - \sum_{k=2}^N \frac{|X_k|^2}{N |\la_1 - \la_k|^2} \right| < c \sum_{k=2}^N \frac{|X_k|^4}{N^2 |\la_1 - \la_k|^4} < \frac{c \gamma^2 (\log N)^2}{ N \delta^{4} },
$$
a deterministic bound which tends to $0$.
We then consider the random variables
$$
Y_k = \frac{|X_k|^2 - \E_X (|X_k|^2) }{|\la_1 - \la_k|^2}
= \frac{|X_k|^2}{|\la_1 - \la_k|^2} - \frac{1}{|\la_1 - \la_k|^2}
$$
which (with respect to $\PP_X$) are independent and centered with uniformly bounded fourth moment; thus their average $\frac1N \sum_{k=1}^N Y_k$ converges almost surely to $0$, $\PP_X$-almost surely. This gives us:
$$
\sum_{k=2}^N \frac{|X_k|^2}{N |\la_1 - \la_k|^2} 
= \sum_{k=2}^{N} \frac{1}{N |\la_1 - \la_k|^2}
+ o_{\PP_X-a.s.}(1)
$$
which, by Lemma \ref{lemma:circular_law}, converges $\PP_{\lambda}$-almost surely to an integral that one can explicitly compute:
$$
\frac{1}{N} \sum_{k=2}^{N} \frac{1}{|\la_1 - \la_k|^2} \xrightarrow[N \rightarrow \infty]{\PP_{\la}-a.s.}
\frac{1}{\pi} \int_{\DD} \frac{1}{|\la_1 - z|^2} \dd^2 z = \ln \left( \frac{|\la_1|^2}{|\la_1|^2 - 1} \right).
$$
This concludes the proof, as with overall probability $1$, $\Ov_{11}$ is arbitrarily close to $\frac{|\la_1|^2}{|\la_1|^2 - 1}$ for large enough $N$.
\end{proof}

Alternatively, this result for a conditioned outlier can be obtained via the approach used by Fyodorov \cite{Fyodorov}, where an exact finite-$N$ expression of the joint density of $\la_1$ and $\Ov_{11}$ is obtained via partial Schur form and supersymmetry; indeed, the asymptotics of the joint density (2.13) yield convergence to a point mass at the value $\frac{|z|^2}{|z|^2-1}$ in the case of an outlier.
One reason to present the above description of the distribution of $\Ov_{11}$ is that the same method can be applied to eigenvalues of the complex elliptic Ginibre ensemble (as was done in \cite{Thomas2026}) and also in principle to other non-Gaussian but integrable ensembles such as the spherical ensemble or the truncated unitary ensembles \cite{DubachSpherical}.

Note that the diagonal overlap of an outlier eigenvalue is still an unstable object and that, for instance, it does not converge in $L^2$; indeed, a quick verification based on formula \eqref{eq_distr} and the joint density \eqref{joint_density_2} shows that $\Ov_{11}$ has infinite variance even when conditioned on $\la_1$ outside the bulk. However, its expectation is finite, and converges to the same value, as we now prove.

\begin{proposition}[Convergence of the conditional expectation of diagonal overlaps]\label{prop_cond_exp_conv}
In the regime where $|z|>1$,
    \begin{equation}
        \E \left( \Ov_{11} \ | \ \la_1 = z \right) = \frac{|z|^2}{|z|^2-1} + O(N^{-1}).
    \end{equation}
\end{proposition}

\begin{proof}
For a complex Ginibre matrix, we know from \cite{Akemann_determinantal,BourgadeDubach,ChalkerMehlig} 
that the conditional expectation of the diagonal overlap is given by 
\begin{equation}
\mathbb{E} \left( \Ov_{11} \ \vert \ \lambda_1 = z \right) 
= N \left( \frac{e^{(N)}(N\vert z\vert^2)}{e^{(N-1)}(N\vert z\vert^2)} - \vert z\vert^2 \right),    
\end{equation}
for any $z \in \CC$. This expression was studied in the bulk ($|z|<1$) where it is asymptotically equivalent to $N (1-|z|^2)$.
In the regime where $|z|>1$, we rely on the relevant asymptotic from Lemma \ref{lemma:asymp_partial_sums} (proved in the appendix) and find, with $t=|z|^2$:
\begin{align*}
    \E \left( \Ov_{11} \ | \ \la_1 = z \right) & = N \left( \frac{e^{(N-1)}(N\vert z\vert^2) + \frac{N |z|^2}{N!}}{e^{(N-1)}(N\vert z\vert^2)} - \vert z\vert^2 \right)  \\
    & = N \left( 1- |z|^2 + \frac{N^N |z|^{2N}}{N! e^{(N-1)} (N|z|^2)} \right) \\
    & = N \left( 1- |z|^2 + \frac{1}{\frac{1}{|z|^2-1} - \frac{1}{N} \frac{|z|^2}{(|z|^2-1)^3} + O(N^{-2})} \right) \\
    & = N \left( 1- |z|^2 + (|z|^2-1) \frac{1}{1 - \frac{1}{N} \frac{|z|^2}{(|z|^2-1)^2} + O(N^{-2})} \right) \\
    & = N \left( 1- |z|^2 + (|z|^2-1) \left( 1 + \frac{1}{N} \frac{|z|^2}{(|z|^2-1)^2} + O(N^{-2}) \right) \right) \\
    & = \frac{|z|^2}{|z|^2-1} + O(N^{-1})
\end{align*}
which yields the claim.
\end{proof}

\newpage
\section{Stability of the outlier of a rank-one perturbation}\label{sec_perturbed}

\noindent In this section, we consider the rank-one perturbation
\begin{equation*}
    G(t):= G + t v v^*.
\end{equation*}
where $G=(G_{ij})_{1\leq i,j \leq N}$ is a complex Ginibre matrix with i.i.d entries $G_{ij}\sim \mathcal{N}_{\CC}(0,1/N)$, and $t>1$ is a fixed perturbation parameter.
By general property of unitary invariance of Ginibre matrices\footnote{Indeed, for any given unitary matrix $U$, $G \stackrel{d}{=} UGU^*$ and so with the relevant choice of $U$, one obtains $G(t) \stackrel{d}{=} G + t e_1 e_1^*$.}, we may assume that $v=e_1$, the first vector of the canonical basis of $\CC^N$, so that only the coefficient $G_{11}$ is affected by the perturbation. It is well known \cites{Chafai, Tao2013} that the empirical spectral measure of $G$ converges almost surely to the circular law $\mu_{\DD}:=\frac{1}{\pi}1_{\lvert z \lvert \leq 1}\dd^2z$ and that the spectral radius satisfies $\rho(G):=\max_{\lambda \in \mathrm{Sp}(G)}\lvert \lambda \lvert \rightarrow1$ almost surely. In particular, the unperturbed Ginibre ensemble has no outlier at the macroscopic scale; the rank-one perturbation $te_1e_1^*$ does not affect the global spectral distribution : the empirical spectral measure $\mu_{G(t)}$ of $G(t)$ still converges almost surely to $\mu_{\DD}$. However, for $t>1$, an outlier eigenvalue appears \cite{Tao2013}. More precisely, almost surely, for $N$ large enough, the spectrum of $G(t)$ consists of exactly one eigenvalue located in a neighborhood of $z=t$, while all the remaining eigenvalues lie inside an arbitrary small neighborhood of the unit disk, as represented on Figure \ref{fig:spectrum_rank1}. We denote the outlier eigenvalue by $\lambda_{1}$. \\

\begin{figure}[h!]
    \centering
    \includegraphics[width=0.9\linewidth]{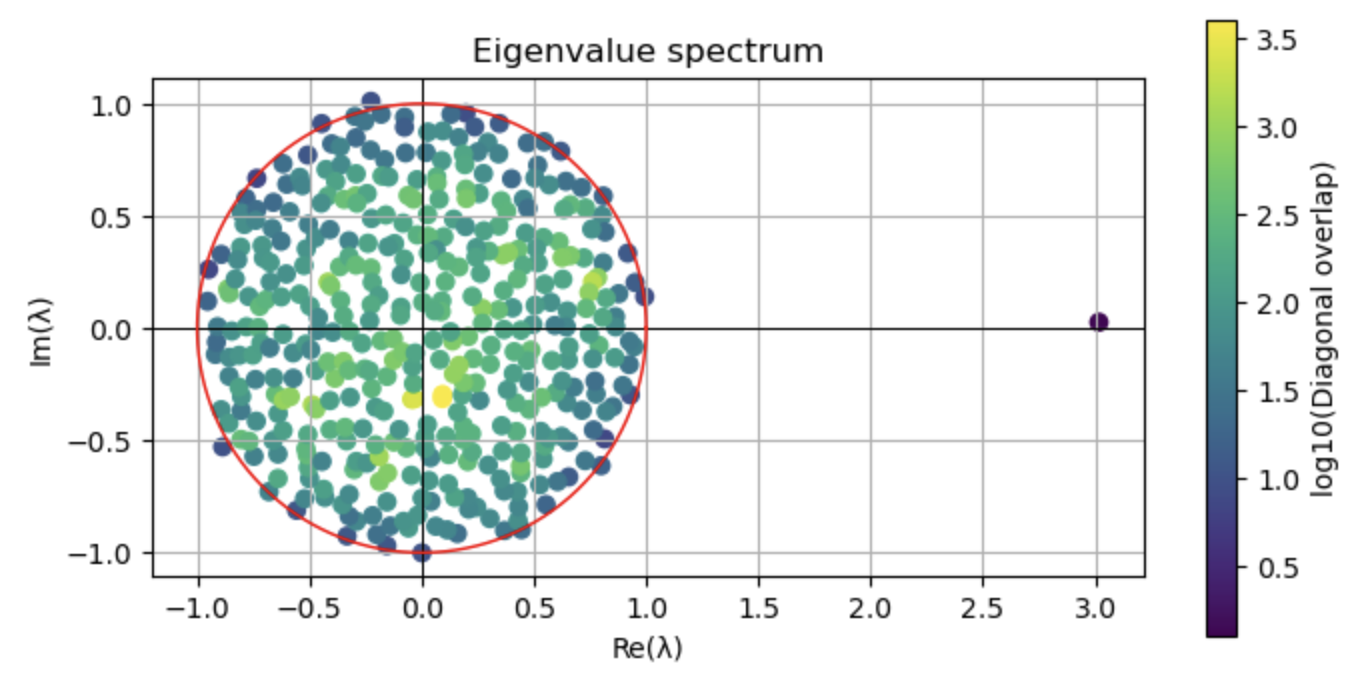}
    \caption{The spectrum of a perturbed Ginibre matrix $G(t) = G+ tvv^*$ with $t = 3$, with the relative size of diagonal overlaps indicated on a color scale. Note that the outlier eigenvalue $\la_1$ is located in a small neighborhood of its expected value, but not exactly at $z=3$.}
    \label{fig:spectrum_rank1}
\end{figure}

Our goal in this section is to determine the limiting behavior of $\Ov_{11}$ as $N\rightarrow \infty$.
By definition we have 
\begin{equation*}
   G= \begin{pmatrix}
        a_N&b_N^*\\
        c_N&D_N
    \end{pmatrix}
\end{equation*}
with $a_N\sim \mathcal{N}_{\C}\left(0,\frac{1}{N}\right), \ b_N,c_N \sim \mathcal{N}_{\C}\left(0,\frac{1}{N}I_{N-1}\right), \ D_N \sim \mathrm{Gin}\left(N-1,\frac{1}{N}\right)$ all independent from each other. \\
\begin{proposition}
    Let $\lambda \notin \mathrm{Sp}(D_N)$, then 
    \begin{equation}\label{eigenvalue_cond}
        \lambda \in \mathrm{Sp}(G(t)) \iff F(\lambda):= \lambda - (a_N+t)- b_N^*\mathbf{R}_{N}(\la)c_N=0
    \end{equation}
    where $\mathbf{R}_N(\lambda):=(\lambda-D_N)^{-1}$
    and the overlap associated to an eigenvalue $\la_1$ is 
    \begin{equation*}
        \Ov_11= \frac{(1+c_N^*\mathbf{R}_{N}(\la_1)^*\mathbf{R}_{N}(\la_1)c_N)(1+b_N^*\mathbf{R}_{N}(\la_1)\mathbf{R}_{N}(\la_1)^*b_N)}{\lvert 1 +b_N^*\mathbf{R}_{N}(\la_1)^2c_N\lvert^2}.
    \end{equation*}
\end{proposition}
\begin{proof}
     The first part of the proposition is a direct application of Schur's complement formula, which yields
    \begin{equation*}
        \mathrm{det}(\lambda-G)=\mathrm{det}(\lambda-D)F(\lambda).
    \end{equation*}
    This proves the first part of the proposition. 
    \\
    Let us look for a right eigenvector of $G$ for the eigenvalue $\la$, of the form 
    $r=\begin{pmatrix}
        x\\y
    \end{pmatrix}$
    with $x\in \C, \ y\in \C^{N-1}$.
    Then
    \begin{equation*}      
            Gr=\lambda r 
        \Longleftrightarrow
        \begin{cases}
            (a_N+t)x+b_N^*y=\lambda x \\
            c_Nx+D_Ny=\lambda y
        \end{cases}
        \Longleftrightarrow
        \begin{cases}
            F(\lambda)=0\\
            y=\mathbf{R}_{N}(\la)c_Nx
        \end{cases}
        \Longleftrightarrow
        y=\mathbf{R}_{N}(\la)c_Nx.
    \end{equation*}
    Taking $x=1$ we get that $r=\begin{pmatrix}
        1\\
        \mathbf{R}_{N}(\la)c_N
    \end{pmatrix}$
    is a right eigenvector of $G$. Likewise one can see that $l^*=\begin{pmatrix}
        1 & b^*\mathbf{R}_{N}(\la)
    \end{pmatrix}$
    is a left eigenvector. Note that $l^*r=1+b_N^*\mathbf{R}_{N}(\la)^2c_N:=s$. 
    Therefore if we define $l_0^*:=\frac{l^*}{s}$ and $r_0:=r$ we have that 
    \begin{equation*}
         \Ov_{11} = \lvert \lvert L_1 \lvert \lvert^2 \lvert \lvert R_1 \lvert \lvert^2 
         =\frac{(1+c_N^*\mathbf{R}_{N}(\la_1)^*\mathbf{R}_{N}(\la_1)c_N)(1+b_N^*\mathbf{R}_{N}(\la_1)\mathbf{R}_{N}(\la_1)^*b_N)}{\lvert 1 +b_N^*\mathbf{R}_{N}(\la_1)^2c_N\lvert^2}
    \end{equation*}
    which concludes the proof.
 \end{proof}
The proof now reduces to the analysis of the three resolvent quantities
\begin{equation*}
    1+b_N^*\mathbf{R}_{N}(z)\mathbf{R}_{N}(z)^*b_N, \qquad
    1+c_N^*\mathbf{R}_{N}(z)^*\mathbf{R}_{N}(z)c_N, \qquad
    b_N^*\mathbf{R}_{N}^2(z)c_N.
\end{equation*}
We first establish almost sure pointwise convergence of the quadratic terms towards 
\begin{equation*}
    q(z):=\frac{\lvert z \lvert^2}{\lvert z \lvert^2-1},
\end{equation*}
using a multiplicative decomposition obtain from the Schur form together with the circular law. The bilinear term is shown to converge pointwise to $0$ using classical probabilistic arguments. \\
These pointwise convergences are then upgraded to uniform convergence on compact subsets of $\{\lvert z \lvert >1\}.$ The key ingredients are a resolvent bound from the literature, yielding uniform control of $\mathbf{R}_{N}(z)$ on compact sets away from the unit disk, and standard resolvent identities, which imply equicontinuity of the quantities under consideration. We can then promote these pointwise convergences to uniform convergence using a finite-net argument.\\
Finally, since the outlier eigenvalue converges almost surely to $t$ with $ t  >1$, it eventually belongs to a fixed compact neighborhood of $t$, $K \subset \{\lvert z \lvert >1\}$, for instance a closed ball with appropriate radius. The uniform convergence results may therefore be evaluated at $z=\la_1$, yielding the claimed asymptotics for the overlap.\\
We begin with the first technical ingredient of the proof, namely a uniform bound on the resolvent outside the unit disk. 
\begin{lemma}
    Let $K \subset \{\lvert z \lvert >1\}$ a compact. Then there exists $C_K>0$ such that almost surely \begin{equation*}
        \sup_{z\in K}\lvert \lvert \mathbf{R}_{N}(z) \lvert \lvert \leq C_K
    \end{equation*} 
    for all sufficiently large $N$.
\end{lemma}
\begin{proof}
    We will prove that there exist $C_K>0$ such that 
    \begin{equation*}
        \inf_{z\in K}s_{\min}(z-D_N)\geq C_K
    \end{equation*}
    which is equivalent to our statement. For $z\in \C$ consider the Hermitian matrix $Y_N^z:=(z-D_N)(z-D_N)^*$. It follows from Proposition 2.2 and equation ($18$a) of \cite{CipolloniErdösSchröder} that the empirical spectral distribution of $Y_N^z$ converges to a deterministic measure supported on an interval $[e_{-}(z), e_{+}(z)]$ such that $e_{-}(z)>0$ whenever $\lvert z \lvert>1.$ Since $K$ is a compact subset of $\{\lvert z \lvert >1\}$, continuity of $e_{-}$ implies that $m_K:=\inf_{z\in K}e_-(z)>0$. Fix $z\in K$. Since the interval $[0,m_K/2]$ lies strictly outside the limiting support of the spectrum of $Y_N^z$, the absence of outliers for $Y_N^z$ (see e.g Theorem 3 of \cite{Vallet} or Theorem 1.1 of \cite{BaiSilverstein}) implies that almost surely, for all $N$ large enough, $\la_{\min}(Y_N^z)\geq m_K/2$. Equivalently, $s_{\min}(z-D_N)\geq \sqrt{m_K/2}$. To obtain a uniform estimate let $\delta:=\frac{1}{2}\sqrt{m_K/2}$ and choose a finite $\delta$-net $\{z_1,\ldots,z_M\}\subset K$. By the previous pointwise estimate, almost surely, for all $N$ large enough,
    \begin{equation*}
        s_{\min}(z_j-D_N)\geq \sqrt{m_K/2}
    \end{equation*}
    for every $j=1,\ldots,M.$ Now using the fact that $s_{\min}$ is $1$-Lipschitz (see for instance Weyl's perturbation theorem in \cite[Cor. III.2.6]{Bhatia})
    \begin{equation*}
        s_{\min}(z-D_N)\geq s_{\min}(z_j-D_N)-\lvert z-z_j\lvert \geq \delta.
    \end{equation*}
    Therefore, almost surely, for all sufficiently large $N$, 
    \begin{equation*}
        \inf_{z \in K}s_{\min}(z-D_N)\geq \delta
    \end{equation*}
    which concludes the proof.
\end{proof}

We now prove the key probabilistic representation underlying the proof of the pointwise convergence.
\begin{lemma}
   Set $Q_N(z):=1+b_N^*\mathbf{R}_{N}(z)\mathbf{R}_{N}(z)^*b_N$ for $z \notin Sp (D_N)$ and denote by $\mu_1,\ldots,\mu_{N-1}$ the eigenvalues of $D_N$. Then 
   \begin{equation*}
       Q_N(z) \overset{}{=} \prod_{k=1}^{N-1}\left(1+\frac{\lvert X_k\lvert^2}{N\lvert z- \mu_k \lvert^2}\right)
   \end{equation*}
   where $X_1,\ldots,X_{N-1}$ are i.i.d. standard complex Gaussian random variables, independent of $\mu_1,\ldots,\mu_{N-1}.$
\end{lemma}
\begin{proof}
    Let $D_N=UTU^*$ be a Schur decomposition of $D_N$, where $T$ is upper triangular and has diagonal entries $\mu_1,\ldots, \mu_{N-1}$. Since $b_N$ is independent of $D_N$ and has a unitary invariant Gaussian $\beta:=U^*b$ has law $\mathcal{N}_{C}(0,N^{-1}I)$ and is independent of $T$ and moreover 
    \begin{equation*}
        b_N^*\mathbf{R}_{N}(z)\mathbf{R}_{N}(z)^*b_N=\beta ^*(z-T)^{-1}(z-T)^{-*}\beta.
    \end{equation*}
    Thus it is enough to prove the identity for $T$ and $\beta$. \\
    Let $T_k $ be the $k\times k$ upper-left principal block of $T$, and let $\beta_k$ be the vector made of the first $k$ coordinates of $\beta.$ Define $R_k(z):=(z-T_k)^{-1}$ and $Q_k(z):=1+\beta_k^* R_k(z)R_k(z)^*\beta_k.$ Then $Q_0(z)=1$ and $Q_{N-1}(z)$ is the quantity of interest. Write 
    \begin{equation*}
        T_{k+1}= \begin{pmatrix}
        T_k&\tau_k\\
        0&\mu_{k+1}
    \end{pmatrix}, \ \beta_{k+1}=\begin{pmatrix}
        \beta_k\\
        \eta_{k+1}
    \end{pmatrix}.
    \end{equation*}
    Conditionally on the spectrum, the vectors $\tau_k$ have independent complex Gaussian entries of variance $1/N$ and are independent of the previous columns. Also, the variables $\eta_k$ are independent complex Gaussians of variance $1/N$, independent of $T$. Set $y_k:=R_k(z)^*\beta_k.$ Then $Q_k(z)=1+ \lvert \lvert y_k \lvert \lvert ^2.$ Write 
    \begin{equation*}
        y_{k+1}=\begin{pmatrix}
        y_k\\
        \gamma_{k+1}
    \end{pmatrix}.
    \end{equation*}
    Since $(z-T_{k+1})^*y_{k+1}=\beta_{k+1}$ and 
    \begin{equation*}
        (z-T_{k+1})^*= \begin{pmatrix}
        (z-T_k)^*&0\\
        -\tau_k^*& \overline{z}-\overline{\mu_{k+1}}
    \end{pmatrix}
    \end{equation*}
    we obtain that 
    \begin{equation*}
        \gamma_{k+1}=\frac{\eta_{k+1}+\tau_k^*y_k}{\overline{z}-\overline{\mu_{k+1}}}.
    \end{equation*}
    Therefore, 
    \begin{equation*}
        Q_{k+1}(z)=1+\lvert \lvert y_{k+1}\lvert \lvert^2=Q_k(z)+ \frac{\lvert \eta_{k+1}+\tau_k^*y_k\lvert^2}{\lvert z-\mu_{k+1}\lvert^2}
    \end{equation*}
    Now define $X_{k+1}:=\sqrt{\frac{N}{Q_k(z)}}(\eta_{k+1}+\tau_k^*y_k).$ Then 
    \begin{equation*}
        Q_{k+1}(z)=Q_k(z)\left(1+ \frac{\lvert X_{k+1}\lvert^2}{N\lvert z-\mu_{k+1}\lvert^2}\right).
    \end{equation*}
    We now identify the law of the $X_k$. Let $\mathcal{F}_k:=\sigma(\mu_1,\ldots,\mu_{N-1},T_k,\beta_k).$ Then $y_k$ and $Q_k(z)$ are $\mathcal{F}_k$-measurable. Moreover, $\eta_{k+1}$ and $\tau_k$ are independent of $\mathcal{F}_k$ and 
    \begin{equation*}
        \eta_{k+1}+\tau_k^*y_k \mid \mathcal{F}_k \sim \mathcal{N}_{\C}\left(0,\frac{1+\lvert \lvert y_k \lvert \lvert^2}{1}\right)=\mathcal{N}_{\C}\left(0,\frac{Q_k(z)}{N}\right).
    \end{equation*}
    Hence $X_{k+1}\mid \mathcal{F}_k \sim \mathcal{N}_{C}(0,1).$ The conditional law does not depend on $\mathcal{F}_k$, therefore $X_{k+1}$ is independent of $\mathcal{F}_k$. In particular, $X_{k+1}$ is independent of the spectrum and of $X_1,\ldots,X_k$. By induction $X_1,\ldots,X_{N-1}$ are i.i.d. standard complex Gaussians independent of $\mu_1,\ldots,\mu_{N-1}$. Iterating the recurrence from $Q_0(z)=1$, we obtain 
    \begin{equation*}
        Q_{N-1}(z)=\prod_{k=1}^{N-1}\left(1+\frac{\lvert X_k \lvert ^2}{N\lvert z-\mu_k \lvert^2}\right)
    \end{equation*}
    which concludes the proof.
\end{proof}
Note that an identical argument applies to $1+c_N^*\mathbf{R}_{N}(z)^*\mathbf{R}_{N}(z)c_N$, yielding the same product representation. We now prove the pointwise convergence of $Q_N(z)$ towards $q(z)$.
\begin{lemma}
    Let $z\in \C$ be such that $\lvert z \lvert >1$. Then 
    \begin{equation*}
        Q_N(z)\xrightarrow[N\rightarrow\infty]{a.s}q(z).
    \end{equation*}
\end{lemma}
\begin{proof}
    Define $a_k(z):=\frac{1}{\lvert z-\mu_k \lvert^2}$ and $S_N(z):=\frac{1}{N}\sum_{k=1}^{N-1}a_k(z)$. We first prove that $\log Q_N(z)-S_N(z)\rightarrow 0$ a.s, then we use the almost sure circular law to get the desired convergence. By Lemma 12 the following equality in law holds true 
    \begin{equation*}
        \log Q_N(z)=\sum_{k=1}^{N-1}\log \left(1+\frac{a_k(z)\lvert X_k \lvert^2}{N}\right).
    \end{equation*}
    Since $\lvert z \lvert >1$ we may choose $\varepsilon>0$ such that $\lvert z\lvert >1+\varepsilon$ and by the almost sure convergence of the spectral radius of $D_N$ to $1$, almost surely, for all $N$ large enough, $\rho(D_N)\leq 1 + \varepsilon/2.$ Therefore for all $k=1,\ldots,N-1$ we have $a_k(z)\leq 4/\varepsilon^2$ almost surely for all $N$ large enough. Using the fact that $0 \leq x-\log(1+x)\leq x^2$ for all $x\geq 0$ we get 
    \begin{equation*}
       \left \lvert \log Q_N(z)-\frac{1}{N}\sum_{k=1}^{N-1}a_k(z)\lvert X_k \lvert^2 \right\lvert \leq \frac{1}{N^2}\sum_{k=1}^{N-1}a_k(z)^2\lvert X_k \lvert^4.
    \end{equation*}
    Since $a_k(z)$ is eventually uniformly bounded and 
    \begin{equation*}
        \frac{1}{N}\sum_{k=1}^{N-1}\lvert X_k \lvert^4 \rightarrow \E\lvert X_k \lvert^4<\infty, \ \textit{almost surely,}
    \end{equation*}
    we obtain that
    \begin{equation*}
        \log Q_N(z) - \frac{1}{N}\sum_{k=1}^{N-1}a_k(z)\lvert X_k \lvert^2 \rightarrow0, \ \textit{almost surely.}
    \end{equation*}
    Then, 
    \begin{equation*}
        \frac{1}{N}\sum_{k=1}^{N-1}a_k(z) \lvert X_k \lvert^2= S_N(z) + \frac{1}{N}\sum_{k=1}^{N-1}a_k(z)(\lvert X_k \lvert^2-1). 
    \end{equation*}
    Since the weights $a_k(z)$ are eventually uniformly bounded, and the variables $\lvert X_k \lvert ^2-1$ are independent, centered, and have finite fourth moment we can use Lemma 4 and deduce that
    \begin{equation*}
        \frac{1}{N}\sum_{k=1}^{N-1}a_k(z)(\lvert X_k \lvert^2-1) \rightarrow0, \ \textit{almost surely}
    \end{equation*}
    and thus $\log Q_N(z) -S_N(z)\rightarrow 0$ almost surely. Now by the almost sure circular law applied to the eigenvalues of $D_N$, 
    \begin{equation*}
        S_N(z)=\frac{1}{N}\sum_{k=1}^{N-1}\frac{1}{\lvert z-\mu_k \lvert^2}\rightarrow\int_{\DD}\frac{1}{\lvert z-u\lvert^2}\frac{d^2u}{\pi}=\log q(z)
    \end{equation*}
    which concludes the proof.
\end{proof}
We now upgrade this pointwise convergence to a uniform one on a compact outside the unit disk.
\begin{proposition}
     Let $K$ be a compact such $K \subset \{\lvert z \lvert >1\}$. Then, a.s,
     \begin{equation*}
         \sup_{z\in K}\lvert Q_N(z)-q(z)\lvert \rightarrow0.
     \end{equation*}
 \end{proposition}
 \begin{proof}
     Let $\mathcal{D}\subset K$ dense countable. Then, a.s, 
     \begin{equation*}
         \forall z \in \mathcal{D}, \ Q_N(z)\rightarrow q(z). 
     \end{equation*}
     Now we work on the almost sure event 
     \begin{equation*}
         \Omega_0:=\{ \forall z \in \mathcal{D}, \ Q_N(z)\rightarrow q(z)\}\cap\{\lvert \lvert b\lvert \lvert \rightarrow1\}\cap \{\textit{for all $N$ large enough} \  \sup_{z\in K}\lvert \lvert \mathbf{R}_{N}(z)\lvert \lvert \leq C_K\}.
     \end{equation*}
     Using the resolvent identity $\mathbf{R}_{N}(z)-\mathbf{R}_{N}(w)=(w-z)\mathbf{R}_{N}(z)\mathbf{R}_{N}(w)$ we get that for all $N$ large enough 
     \begin{equation*}
         \lvert \lvert \mathbf{R}_{N}(z)-\mathbf{R}_{N}(w) \lvert \lvert \leq C_K^2 \lvert z-w \lvert.
     \end{equation*}
     We denote $A(z):=\mathbf{R}_{N}(z)\mathbf{R}_{N}(z)^*$. Then noting that 
     \begin{equation*}
         A(z) -A(w)= (\mathbf{R}_{N}(z)-\mathbf{R}_{N}(w))\mathbf{R}_{N}(z)^* + \mathbf{R}_{N}(w)(\mathbf{R}_{N}(z)^*-\mathbf{R}_{N}(w)^*)
     \end{equation*}
     we obtain that 
     \begin{equation*}
         \lvert \lvert A(z) -A(w) \lvert \lvert \leq 2C_K^3 \lvert z-w\lvert.
     \end{equation*}
     Now for all $N$ large enough we have $\lvert \lvert b \lvert \lvert \leq 2$. Hence 
     \begin{align*}
         \lvert Q_N(z)-Q_N(w)\lvert &=\lvert b^*(A(z)-A(w))b\lvert \\
         &\leq \lvert \lvert b\lvert \lvert^2 \lvert \lvert A(z)-A(w)\lvert \lvert \\
         &\leq 8C_K^3 \lvert z -w \lvert.
     \end{align*}
     Therefore the $(Q_N)$ are equi-Lipschitz on $K$, a.s, for all $N$ large enough. The limit function $q$ is continuous, so uniformly continuous on $K$. Fix $\varepsilon>0$ and let $\eta>0$ such that $8C_K^3\eta\leq \frac{\varepsilon}{3}$ and $\lvert q(z)-q(w) \lvert \leq \frac{\varepsilon}{3}$ whenever $\lvert z-w\lvert \leq \eta$.
     Since $\mathcal{D}$ is dense and $K$ is a compact, there exists a finite $\eta$-net meaning $z_1,\ldots,z_M\in \mathcal{D}$ such that for all $z\in K$ there exists a $j$ such that $\lvert z- z_j\lvert \leq \eta.$ For all $N$ large enough we have 
     \begin{equation*}
         \max_{1\leq j\leq M}\lvert Q_N(z_j)-q(z_j) \lvert \leq \frac{\varepsilon}{3}.
     \end{equation*}
     Let $z \in K$ and $z_j$ such that $\lvert z-z_j \lvert \leq \eta$, we then have 
     \begin{align*}
         \lvert Q_N(z) -q(z)\lvert &\leq \lvert Q_N(z) - Q_N(z_j)\lvert + \lvert Q_N(z_j)-q(z_j)\lvert + \lvert q(z_j)-q(z) \lvert \\
         &\leq \frac{\varepsilon}{3}+\frac{\varepsilon}{3}+\frac{\varepsilon}{3}.
     \end{align*}
     Therefore, a.s, for all $N$ large enough 
     \begin{equation*}
         \sup_{z\in K}\lvert Q_N(z)-q(z) \lvert \leq \varepsilon
     \end{equation*}
     which concludes the proof.
 \end{proof}
By the same argument one also obtains $\sup_K \lvert 1+ c_N^*\mathbf{R}_{N}(z)^*\mathbf{R}_{N}(z)c_N-q(z) \lvert \rightarrow 0$ almost surely.\\
We now prove the convergence towards $0$ of the bilinear term $b_N^*\mathbf{R}_{N}(z)c_N$.
\begin{proposition}
     Let $K$ be a compact such $K \subset \{\lvert z \lvert >1\}$. Then, a.s,
     \begin{equation*}
         \sup_{z\in K}\lvert b_N^*\mathbf{R}^2_{N}(z)c_N\lvert \rightarrow0.
     \end{equation*}
\end{proposition}
\begin{proof}
     Let $z \in K$ fixed. We first prove that $Y_N(z)\rightarrow0$ a.s. Define $\mathcal{F}:=\sigma((D_N,b_N)_N)$. Conditionally on $\mathcal{F}, \ c_N \sim \mathcal{N}_{\C}(0,N^{-1}I)$ and $Y_N(z)$ is a centered Gaussian of variance $\sigma_N^2(z)=\frac{1}{N}b_N^*B_N(z)B_N(z)^*b_N$ where $B_N(z):=\mathbf{R}^2_{N}(z)$. Now, a.s, for all $N$ large enough we have 
     \begin{equation*}
         \sigma_N^2(z)\leq \frac{1}{N}\lvert \lvert b_N\lvert \lvert ^2 \lvert \lvert B_N(z)\lvert \lvert^2\leq \frac{C}{N}.
     \end{equation*}
     Therefore for all $\delta>0$ we have, a.s, for all $N$ large enough, 
     \begin{equation*}
         \PP(\lvert Y_N(z)\lvert> \delta \ \lvert \ \mathcal{F})\leq \exp(-CN). 
     \end{equation*}
     Consequently,a.s,
     \begin{equation*}
         \sum_N \PP(\lvert Y_N(z)\lvert> \delta \ \lvert \ \mathcal{F}) < \infty.
     \end{equation*}
     Therefore $\PP(\lvert Y_N(z)\lvert> \delta  \ \textit{i.o}\ \lvert \ \mathcal{F})=0$  and then by tower property $\PP(\lvert Y_N(z)\lvert> \delta  \ \textit{i.o})=0$. Hence a.s $Y_N(z)\rightarrow0$. Now to extend this to an almost sure uniform convergence on $K$ notice that 
     \begin{equation*}
         \mathbf{R}^2_{N}(z)-\mathbf{R}^2_{N}(w)=\mathbf{R}^2_{N}(z)(\mathbf{R}^2_{N}(z)-\mathbf{R}^2_{N}(z))+(\mathbf{R}^2_{N}(z)-\mathbf{R}^2_{N}(w))\mathbf{R}_{N}(w)
     \end{equation*}
     and therefore, a.s, for all $N$ large enough
     \begin{equation*}
         \lvert \lvert \mathbf{R}^2_{N}(z)-\mathbf{R}^2_{N}(w) \lvert \lvert \leq 2C_K^3\lvert z-w \lvert.
     \end{equation*}
     In particular, a.s, for all $N$ large enough
     \begin{equation*}
         \lvert Y_N(z) -Y_N(w)\lvert =\lvert b_N^*(\mathbf{R}^2_{N}(z)-\mathbf{R}^2_{N}(z))c_N\lvert \leq \lvert \lvert b_N\lvert\lvert \ \lvert \lvert c_N \lvert \lvert  \ 2C_K^3\lvert z-w \lvert \leq 8C_K^3 \lvert z-w \lvert.
     \end{equation*}
     Fix $\varepsilon>0$ and $\eta$ such that $8C_K^3 \eta \leq \frac{\varepsilon}{2}$. Let $z\in K$  and let $j$ such that $\lvert z-z_j\lvert \leq \eta$. We then have, a.s, for all $N$ large enough
     \begin{equation*}
         \lvert Y_N(z) \lvert \leq \lvert Y_N(z_j)\lvert + \lvert Y_N(z)-Y_N(z_j)\lvert\leq \frac{\varepsilon}{2}+ \frac{\varepsilon}{2}.
     \end{equation*}
     Hence, a.s, for all $N$ large enough, $\sup_K \lvert Y_N(z)\lvert \leq \varepsilon$ which concludes the proof.
     \end{proof}
     We can now combine the previous estimates to obtain the almost sure convergence of the diagonal overlap associated with the outlier eigenvalue.
\begin{theorem}
    Let $G(t)=G+te_1e_1^*$ where $t>1$ and $G$ is a complex Ginibre matrix with entries of variance $1/N$. Let $\la_1$ be the outlier eigenvalue of $G(t)$. Then 
    \begin{equation*}
        \Ov_{11}\xrightarrow[N\rightarrow\infty]{a.s}\left(\frac{t^2}{t^2-1}\right)^2.
    \end{equation*}
\end{theorem}
\begin{proof}
    Let $K:=D(t,\varepsilon)$ with $\varepsilon$ being small enough such that $K\subset \{\lvert z \lvert >1\}$. Now almost surely, for all $N$ large enough 
    \begin{equation*}
        \la_1 \in K \ \textit{and} \ \mathrm{Sp}(D_N)\cap K=\emptyset.
    \end{equation*}
    By Proposition $10$
    \begin{equation*}
        \Ov_{11}= \frac{(1+c_N^*\mathbf{R}_{N}(\la_1)^*\mathbf{R}_{N}(\la_1)c_N)(1+b_N^*\mathbf{R}_{N}(\la_1)\mathbf{R}_{N}(\la_1)^*b_N)}{\lvert 1 +b_N^*\mathbf{R}_{N}(\la_1)^2c_N\lvert^2}.
    \end{equation*}
    Now by the previous uniform convergence results, almost surely, for all $N$ large enough
    \begin{equation*}
        \Ov_{11}=(q(\la_1)+o(1))^2+o(1).
    \end{equation*}
    Since $q$ is continuous on $K$ and $\la_1 \rightarrow t$ almost surely 
    \begin{equation*}
        q(\la_1)\rightarrow q(t)=\frac{t^2}{t^2-1}
    \end{equation*}
    which concludes the proof.
\end{proof}

\section{Conclusion: some comments on methods and phenomena}\label{sec_conclusion}

We have studied the distribution of diagonal overlaps for outliers with two slightly different techniques. In the first case, we have used the decomposition of the law coming from the Schur form, and in the second case we have used Schur complements to obtain a characterization of the outlier, and uniform bounds that hold with high probability outside the bulk. Each method is taylored to the parameters of the problem: indeed, the first technique cannot easily be applied to the second case, because the distribution of the Schur form of a perturbed matrix is not so simple to characterize; and the second technique cannot be used for the first problem, because conditioning on an event of measure zero tilts the distributions and jeopardizes the uniform bounds (which only hold on an event of high probability -- but then, presumably not under the conditioning).  This is also the reason why there is no contradiction between the two results, which describe two related but distinct regimes. 

All recent works on overlaps and their distribution in integrable ensembles suggest an interpretation of the dynamical version of the `freezing at the edge' phenomenon \cite{FreezingEdge}, and an answer to this natural question: why are eigenvalues more stable at the edge than inside the bulk? The key insight seems to be that the diagonal overlaps (which always encapsulates the stability of the associated eigenvalues, as shown by Proposition \ref{prop1} as well as other related results in the literature \cites{BourgadeDubach,GrelaWarchol,ChalkerMehlig}) depend multiplicatively on the distance between eigenvalues, with the most relevant contributions coming from the closest neighbors. This is exactly true for the complex Ginibre ensemble, and can be expected to still be the case at leading order for a variety of reasonable distributions. As a consequence, eigenvalue instability is roughly a multiplicative quantity of neighboring eigenvalues, whose overall scale depends chiefly on the behavior of nearest neighbors. The difference between the bulk and the edge, at microscopic level, being essentially the fact that the nearest neighbors tend to occupy a half-space rather than the whole neighborhood, it is to be expected as a rule of thumb that the diagonal overlaps at the edge scale like the square root of those in the bulk. In the context of the complex Ginibre ensemble, this heuristic is exactly true, with diagonal overlaps having order $N$ in the bulk, $\sqrt{N}$ at the edge, and order $1$ for outliers as established in this paper. These scales reflect the fact that an edge eigenvalue has roughly half as much close neighbors compared to a bulk eigenvalue, and that an outlier eigenvalue typically has no close neighbor at all. 

\newpage
\section{Appendix: asymptotics of the incomplete Gamma function}\label{appendix}

\subsection{Estimates on the partial sums of the exponential}\label{appendix_A}

We recall the definition of the incomplete Gamma function:
$$\Gamma(N, x) = \int_{x}^{\infty} s^{N-1} e^{-s} ds.$$
This function appears naturally when dealing with statistics of the complex Ginibre ensemble; in Section~\ref{sec_conditional}, we rely on a precise asymptotic expansion of the incomplete Gamma function, specifically in the range $x = Nt$, with $t>1$, corresponding to the squared radius of an outlier.

\begin{lemma}[Asymptotic expansion of $\Gamma(N,Nt)$] For any fixed $t>1$, the following estimate holds
$$\Gamma(N,Nt) = (Nt)^{N-1} e^{-Nt} \left( \frac{t}{t-1} - \frac{1}{N} \frac{t^2}{(t-1)^3} + O(N^{-2})\right).$$
\end{lemma}

\begin{proof}
A change of variable $Nu = s$ gives
$$\Gamma(N, Nt) = N^N \int_{t}^{\infty} u^{N-1} e^{- N u} du = N^N  \int_{t}^{\infty} u^{-1} e^{ N ( \ln(u) - u ) } du.$$
Let us denote $\varphi(u) = \ln(u) -u$, $\varphi'(u) = \frac{1-u}{u}$ and perform an integration by parts:

$$N^{-N} \Gamma(N, Nt) = \int_{t}^{\infty} \frac{1}{1-u} \frac{1}{N} \partial \left( e^{ N ( \ln(u) - u ) } \right) du
= \frac{1}{N} \frac{1}{t-1} e^{N \varphi(t)} - \frac{1}{N} \int_{t}^{\infty} \frac{1}{(1-u)^2}  e^{ N ( \ln(u) - u ) } du,$$
which isolates in fact the first term of the asymptotic expansion; one more integration by part yields:
\begin{align*}
    \int_{t}^{\infty} \frac{1}{(1-u)^2}  e^{ N ( \ln(u) - u ) } du &= \frac{1}{N}\int_{t}^{\infty} \frac{u}{(1-u)^3}  \partial  \left( e^{ N ( \ln(u) - u ) } \right) du \\
    &= \frac{1}{N} \frac{t}{(t-1)^3} e^{ N ( \ln(t) - t ) } - \frac{1}{N}\int_{t}^{\infty} \partial \left( \frac{u}{(1-u)^3} \right)  e^{ N ( \ln(u) - u ) } du,
\end{align*}
which gives us the second term. To bound the last term, we notice first that 
$$\frac{d}{du} \left( \frac{u}{(1-u)^3} \right) = \frac{1 + 2u}{(1-u)^4} = \frac{2}{(u-1)^3} + \frac{3}{(u-1)^4},$$ 
which is bounded by
$$C(t) = \frac{2}{(t-1)^3} + \frac{3}{(t-1)^4}$$
as $u \geq t > 1$. We then study the remaining integral term.
\begin{align*}
    \int_{t}^{\infty} u^N e^{-Nu} du &= e^{-Nt}\int_{0}^{\infty} (v+t)^N e^{-Nv} dv\\
    &= N^{-1} e^{-Nt}\int_{0}^{\infty} (t + \frac{w}{N})^N e^{-w} dw\\
    &= N^{-1} e^{N(ln(t) -t)}\int_{0}^{\infty} (1 + \frac{w}{Nt})^N e^{-w} dw\\
    &\leq N^{-1} e^{N(ln(t) -t)}\int_{0}^{\infty} e^{w(\frac{1}{t}-1)}dw
\end{align*}
Since $t>1$ in our regime, the last term is the integral of a negative exponential : it is a positive constant $K(t)$. We have 
$$\left| - \frac{1}{N}\int_{t}^{\infty} \frac{d}{du} \left( \frac{u}{(1-u)^3} \right)  e^{ N ( \ln(u) - u ) } du \right| \leq C(t) K(t) \frac{1}{N^2} e^{N(ln(t) -t)}$$
Hence,
\begin{align*}
    \Gamma(N, Nt) &= N^{N}\left( \frac{1}{N} \frac{1}{t-1} e^{N(ln(t)-t)} - \frac{1}{N} (\frac{1}{N} \frac{t}{(t-1)^3} e^{ N ( \ln(t) - t ) } - e^{N(ln(t) -t)}O(\frac{1}{N^2} ))\right)\\
    &= N^{N-1} t^{N} e^{-Nt}\left(  \frac{1}{t-1} - \frac{1}{N} \frac{t}{(t-1)^3}  + O(N^{-2} )\right)
\end{align*}
which is the claim.
\end{proof}

\begin{lemma}[Asymptotic behavior of partial sums of the exponential]\label{lemma:asymp_partial_sums}
In the regime $t>1$: 
$$\frac{N!}{(Nt)^N} e^{(N-1)} (Nt) =  \frac{1}{t-1} - \frac{1}{N} \frac{t}{(t-1)^3} + O(N^{-2})$$
\end{lemma}

\begin{proof}
We first write classically 
$$
e^{(N-1)} (x)
= \frac{e^x}{(N-1)!} \int_{x}^{\infty} s^{N-1} e^{-s} d s
= \frac{e^x}{(N-1)!} \Gamma(N,x).
$$
So in the regime $x=Nt$ with $t>1$, we find
$$
e^{(N-1)} (Nt) = \frac{e^{Nt}}{(N-1)!} \Gamma(N,Nt) = \frac{ N^{N-1}t^N }{(N-1)!}   \left( \frac{1}{t-1} - \frac{1}{N} \frac{t}{(t-1)^3} + O(N^{-2})\right)
$$
and finally
$$
\frac{N!}{(Nt)^N} e^{(N-1)} (Nt) =  \frac{1}{t-1} - \frac{1}{N} \frac{t}{(t-1)^3} + O(N^{-2}).
$$

\end{proof}

\end{document}